\documentclass[12pt]{article}
\usepackage{makeidx}
\usepackage{amssymb}
\usepackage{amsfonts}
\usepackage{graphicx}
\usepackage{amsmath}

\newtheorem{theorem}{Theorem}

\newtheorem{proposition}[theorem]{Proposition}
\newtheorem{remark}[theorem]{Remark}

\newenvironment{proof}[1][Proof]{\textbf{#1.} }{\ \rule{0.5em}{0.5em}}

\begin{document}

\date{}
\title{Preconditioning linear systems via matrix function evaluation}
\author{P. Novati, M. Redivo-Zaglia\thanks{%
Corrisponding author. Email adresses: \texttt{novati@math.unipd.it} (P.
Novati), \texttt{Michela.RedivoZaglia@unipd.it} (M. Redivo-Zaglia), \texttt{%
mrrusso@math.unipd.it} (M.R. Russo)}, M.R. Russo \\
Department of Pure and Applied Mathematics\\
University of Padova, Italy }
\maketitle

\begin{abstract}
For the solution of discrete ill-posed problems, in this paper a novel
preconditioned iterative method based on the Arnoldi algorithm for matrix
functions is presented. The method is also extended to work in connection
with Tikhonov regularization. Numerical experiments arising from the
solution of integral equations and image restoration are presented.

\vspace{0.3cm} \noindent\textbf{Keywords:} Tikhonov regularization, Matrix
functions, Arnoldi method.
\end{abstract}

\section{Introduction}

In this paper we consider the solution of ill-conditioned linear systems%
\begin{equation*}
A{\mathbf{x}}={\mathbf{b}},
\end{equation*}%
in which we assume $A\in \mathbb{R}^{N\times N}$ to be full rank with
singular values that gradually decay to $0$. As reference problems we
consider the linear systems arising from the discretization of Fredholm
integral equation of the first kind (commonly referred to as discrete
ill-posed problems \cite{PCH}), where $A$ represents the discretization of a
compact operator. Most of the arguments here presented can also be applied
to certain saddle point problems (see e.g. \cite{BGL}) or even Vandermonde
type systems arising from interpolation theory (see e.g. \cite{Fas}). For
important applications, involving for instance Vandermonde type systems, ${%
\mathbf{b}}$ is assumed to be error-free. On the other hand, working with
discrete ill-posed problems one typically assumes the right-hand side ${%
\mathbf{b}}$ affected by noise. In this paper we consider both cases, taking
into account the two possible situations.

In this framework, it is well known that many Krylov type methods such as
the CG and the GMRES possess certain regularizing properties that allow to
consider them as effective alternative to the popular Tikhonov
regularization method, based on the minimization of the functional%
\begin{equation}
J({\mathbf{x}},\lambda )=\left\Vert A{\mathbf{x}}-{\mathbf{b}}\right\Vert
^{2}+\lambda \left\Vert H{\mathbf{x}}\right\Vert ^{2},  \label{mn}
\end{equation}%
($\left\Vert \cdot \right\Vert $ denoting the Euclidean vector norm) where $%
\lambda >0$ is a given parameter and $H$ is a regularization matrix (see
e.g. \cite{PCH} and \cite{HH} for a background). Indeed, since most of
Krylov methods working with $A$ or $A^{T}A$ initially pick up the largest
singular values of $A$, they can be interpreted as regularization methods in
which the regularization parameter is the iteration number $m$. We may refer
to the recent paper \cite{BSS} and the reference therein for an analysis of
the spectral approximation properties of the Arnoldi-based methods and again
\cite{PCH} \S 6 for the CG-like methods. Anyway, in the framework of
discrete ill-posed problems, Krylov subspace methods also present some
important drawbacks. First of all we may have semi-convergence, that is, the
method initially converges but rather rapidly diverges. This phenomenon
typically appears when the Krylov method is implemented with the
re-orthogonalization of the Krylov vectors (as for instance in the case of
the Matlab version of the GMRES, where the orthogonality of the Krylov basis
is guaranteed at the machine precision by the use of the Householder
transformations). In this situation, after approximating the larger singular
values (oversmoothing) the method is also able to provide a good
approximation to the smallest ones (undersmoothing). This allows to reach
the maximum accuracy, attained for a certain $m_{\mathrm{opt}}$, but at the
same time a reliable stopping criterium needs to be used to avoid
divergence. On the other side, if a Krylov method is implemented without
re-orthogonalization it is typically not able to produce good approximation
of the smallest singular values. After say $\overline{m}$ iterations
(normally with $\overline{m}<m_{\mathrm{opt}}$, hence in a situation of
oversmoothing) multiple or spurious approximations of the smallest singular
values typically appears because of the loss of orthogonality, and the
iteration stagnates around ${\mathbf{x}}_{\overline{m}}$. In this situation
a valid stopping rule is no more so crucial but unfortunately the attainable
accuracy is generally much poorer than the one obtained by the same method
with re-orthogonalization. We refer to \cite{PCH} \S 6.7 for an exhaustive
explanation about the influence of re-orthogonalization in some classical
Krylov methods.

In order to overcome these problems, in this paper we present a new method
that can be referred to as a preconditioned iterative solver in which the
preconditioner is $\left( A+\lambda I\right) $ or $(A^{T}A+\lambda H^{T}H)$.
In detail, in the noise-free case, the method is based on the solution of
the regularized system%
\begin{equation*}
\left( A+\lambda I\right) {\mathbf{x}}_{\lambda }={\mathbf{b}},
\end{equation*}%
and then on the computation of the solution ${\mathbf{x}}$ as%
\begin{equation}
{\mathbf{x}}=f(A){\mathbf{x}}_{\lambda },  \label{mf}
\end{equation}%
where $f(z)=1+\lambda z^{-1}$, using the standard Arnoldi method for matrix
functions based on the construction of the Krylov subspaces with respect to $%
A$ and ${\mathbf{x}}_{\lambda }$, that is, $K_{m}(A,{\mathbf{x}}_{\lambda })=%
\mathrm{span}\{{\mathbf{x}}_{\lambda },A{\mathbf{x}}_{\lambda },...,A^{m-1}{%
\mathbf{x}}_{\lambda }\}$. The method can be viewed as a preconditioned
iterative method, since $f(A)=A^{-1}(A+\lambda I)$. We have to remember that
a regularization of the type $(A+\lambda H){\mathbf{x}}_{\lambda }={\mathbf{b%
}}$ has been considered by Franklin in \cite{F} when $A$ is SPD.

It is worth noting that, with respect to standard preconditioned Krylov
methods, in our method only one system with the preconditioner has to be
solved so reducing the computational cost. Moreover it is important to point
out that for problems in which the singular values of $A$ rapidly decay to
0, as those considered in this paper, each Krylov method based on $A$ shows
a superlinear convergence (see \cite{Ne} Chapter 5). For our method, this
fast convergence is preserved since we still work with $A$ for the
computation of (\ref{mf}) (see Section \ref{sec3} for details). As we shall
see, this idea, i.e., first regularize then reconstruct, will allow to solve
efficiently the problem of divergence without loosing accuracy with respect
to the most effective solvers.

The method can be extended to problems in which the right hand side ${%
\mathbf{b}}$ is affected by noise just considering as preconditioner the
matrix $(A^{T}A+\lambda H^{T}H)$ (cf. (\ref{mn})). As before the idea is to
solve the system%
\begin{equation*}
(A^{T}A+\lambda H^{T}H){\mathbf{x}}_{\lambda }=A^{T}{\mathbf{b}},
\end{equation*}%
and then to approximate the solution ${\mathbf{x}}$ by means of a matrix
function evaluation%
\begin{equation*}
f(Q){\mathbf{x}}_{\lambda }=\left( A^{T}A\right) ^{-1}(A^{T}A+\lambda H^{T}H)%
{\mathbf{x}}_{\lambda },
\end{equation*}%
where $f$ is as before and $Q=\left( H^{T}H\right)^{-1}\!\! \left(
A^{T}A\right) $.

We need to point out that we could unify the theory taking $H=I$ for the
noise-free case, and hence work always with the Krylov subspaces with
respect to the matrix $Q$. However, since $A$ is ill-conditioned, for
evident reasons, we prefer to consider two separate situations. Thus, we
shall denote by ASP (Arnoldi with Shift Preconditioner) and ATP (Arnoldi
with Tikhonov Preconditioner) the two approaches, respectively, for
noise-free and noisy problems respectively.

Besides the stability and the good accuracy, there is a third important
property that holds in both cases: the reconstruction phase, that is, the
matrix function computation, allows to select initially the parameter $%
\lambda $ even much larger (heavy oversmoothing) than the one considered
optimal by the standard parameter-choice analysis (L-curve, Discrepancy
Principle,..., see \cite{PCH} for a background), without important changes
in terms of accuracy. In this sense the method can be considered somehow
independent of the parameter $\lambda $ (see the filter factor analysis
presented in Section \ref{sec4}).

We remark that the idea of using matrix function evaluations to improve the
accuracy of the regularization of ill-conditioned linear systems has already
been considered in \cite{RDLS}. Anyway it is important to point out that the
approach here presented is completely different since, as said before, only
one regularized system needs to be solved. Indeed, in \cite{RDLS} the
authors considers approximations belonging to the Krylov subspaces generated
by $\left( A+\lambda I\right) ^{-1}$ or $(A^{T}A+\lambda H^{T}H)^{-1}$
(rational Krylov approach) that requires the solution of a regularized
linear system at each Krylov step. Here we consider polynomial type
approximations.

The paper is structured as follows. In Section \ref{sec2} we provide a
background about the basic features of the Arnoldi method for matrix
functions and we present the methods (ASP and ATP) studied in the paper. In
Section \ref{sec3} we analyze the error of the ASP method, providing also
some consideration about the error of both methods in inexact arithmetic. In
Section \ref{sec4} we analyze the filter factors of the methods. In Section %
\ref{sec5} we present some numerical experiments, and a test of image
restoration is shown in Section \ref{sec6}. Some final comments are given in
Section \ref{sec7}.

\section{The ASP and the ATP methods}

\label{sec2} As already partially explained in the introduction, the ASP
method approximates the solution of the ill-conditioned system $\ A{\mathbf{x%
}}={\mathbf{b}}$ in two steps, first solving in some way the regularized
system
\begin{equation}
\left( A+\lambda I\right) {\mathbf{x}}_{\lambda }={\mathbf{b}},
\label{sysf1}
\end{equation}%
and then recovering the solution ${\mathbf{x}}$ from the system
\begin{equation*}
\left( A+\lambda I\right)^{-1} \! \!A{\mathbf{x}}={\mathbf{x}}_{\lambda },
\end{equation*}%
that is equivalent to compute
\begin{equation}
{\mathbf{x}}=f(A){\mathbf{x}}_{\lambda }  \label{fa}
\end{equation}%
where%
\begin{equation}
f(z)=1+\lambda z^{-1}.  \label{fz}
\end{equation}%
Independently of the way we intend to approximate ${\mathbf{x}}$ from (\ref%
{fa}), this represents a novel approach because contrary to standard
preconditioned iterative methods here only one linear system with the
preconditioner needs to be solved. Of course this is possible because of the
special preconditioner we are using but, in principle, the idea can be
extended to any polynomial preconditioner.

For the computation of $\ f(A){\mathbf{x}}_{\lambda }$ we use the standard
Arnoldi method (or Lanczos in the symmetric case) projecting the matrix$\ A$%
\ onto the Krylov subspaces generated by $A$ and ${\mathbf{x}}_{\lambda }$,
that is $K_{m}(A,{\mathbf{x}}_{\lambda })=\mathrm{span}\{{\mathbf{x}}%
_{\lambda },A{\mathbf{x}}_{\lambda },...,A^{m-1}{\mathbf{x}}_{\lambda }\}$.
For the construction of the subspaces $K_{m}(A,{\mathbf{x}}_{\lambda })$,
the Arnoldi algorithm generates an orthonormal sequence$\ \left\{ {\mathbf{v}%
}_{j}\right\} _{j\geq 1}$, with ${\mathbf{v}}_{1}={\mathbf{x}}_{\lambda
}/\left\Vert {\mathbf{x}}_{\lambda }\right\Vert $, such that $K_{m}(A,{%
\mathbf{x}}_{\lambda })=\mathrm{span}\left\{ {\mathbf{v}}_{1},{\mathbf{v}}%
_{2},...,{\mathbf{v}}_{m}\right\} $ (here and below the norm used is always
the Euclidean norm). For every $m$, in matrix formulation, we have
\begin{equation}
AV_{m}=V_{m}H_{m}+h_{m+1,m}{\mathbf{v}}_{m+1}{\mathbf{e}}_{m}^{T},
\label{cla}
\end{equation}%
where $V_{m}=\left[ {\mathbf{v}}_{1},{\mathbf{v}}_{2},...,{\mathbf{v}}_{m}%
\right] $, $H_{m}$ is an upper Hessenberg matrix with entries $h_{i,j}={%
\mathbf{v}}_{i}^{T}A{\mathbf{v}}_{j}$ and ${\mathbf{e}}_{j}$ is the $j$-th
vector of the canonical basis of \ $\mathbb{R}^{m}$.

The $m$-th Arnoldi approximation to ${\mathbf{x}}=f(A){\mathbf{x}}_{\lambda }
$ is defined as
\begin{equation}
{\mathbf{x}}_{m}=\left\Vert {\mathbf{x}}_{\lambda }\right\Vert V_{m}f(H_{m}){%
\mathbf{e}}_{1},  \label{asp}
\end{equation}%
(see \cite{H} and the references therein for a background). For the
computation $f(H_{m})$, since the method is expected to produce a good
approximation of the solution in a relatively small number of iterations
(see Section \ref{sec3}), that is for $m\ll N$, one typically considers a
certain rational approximation to $f$, or, as in our case, the Schur-Parlett
algorithm, \cite{H} Chapter 9.

We denote by ASP method the iteration (\ref{asp}) independently of the
method chosen for solving (\ref{sysf1}). Starting from ${\mathbf{v}}_{1}={%
\mathbf{x}}_{\lambda }/\left\Vert {\mathbf{x}}_{\lambda }\right\Vert $, at
each step of the Arnoldi algorithm, we only have to compute the vectors ${%
\mathbf{w}}_{j}=A{\mathbf{v}}_{j}$, $j\geq 1$. Below the algorithm used to
implement the method.

\begin{center}
\hrulefill \\[0pt]
\vspace{-0.1cm} \textbf{ASP Algorithm } \\[0pt]
\vspace{-0.2cm} \hrulefill \\[0pt]
\begin{tabbing}
NN \= XX \= XX \= XX \= XX \= \kill
{\bf Require} $A\in {\mathbb{R}}^{N\times N},\;{\mathbf{b}}\in {\mathbb{R}}^{N}$, $ \lambda\in {\mathbb{R^{+}}}$ \\
{\bf Define} $f(z) = 1+\lambda z^{-1}$\\
{\bf Solve} $\left( A+\lambda I\right) {\mathbf{x}}_{\lambda }={\mathbf{b}}$\\
 ${\mathbf{v}}_{1} \leftarrow {\mathbf{x}}_{\lambda }/\Vert {\mathbf{x}}_{\lambda }\Vert $\\
{\bf for} $m=1,2,\ldots$ {\bf do}\\
\> ${\mathbf{w}}_{m} \leftarrow A{\mathbf{v}}_{m}$\\
\>   $h_{k,m} \leftarrow {\mathbf{v}}_{k}^{T}{\mathbf{w}}_{m}$\\
\>  $\widetilde{{\mathbf{v}}} \leftarrow {\mathbf{w}}_{m}-\sum\nolimits_{k=1}^{m}h_{k,m}{\mathbf{v}}_{k}$\\
\>  $h_{m+1,m} \leftarrow \Vert \widetilde{{\mathbf{v}}}\Vert $ \\
\>  ${\mathbf{v}}_{m+1} \leftarrow \widetilde{{\mathbf{v}}}%
/h_{m+1,m}$\\
\> {\bf Compute} $f(H_{m})$ by Schur-Parlett algorithm\\
\> ${\mathbf{x}}_{m} \leftarrow \Vert {\mathbf{x}}_{\lambda }\Vert V_{m}f(H_{m}){\mathbf{e}}_{1}$\\
{\bf end for}
\end{tabbing}
\vspace{-0.4cm} \hrulefill
\end{center}

In the above algorithm, the Arnoldi method is implemented with the modified
Gram-Schmidt process. Therefore, as is well known, the theoretical
orthogonality of the basis is lost quite rapidly and consequently the method
is not able to pick up the singular values clustered near 0. For this reason
at a certain point during the iteration (\ref{asp}) the method is no longer
able to improve the quality of the approximation and it stagnates, typically
quite close to the best attainable approximation, and almost independently
of the choice $\lambda $ (see Section \ref{sec5}).

Regarding the attainable accuracy (assuming that the seed ${\mathbf{x}}%
_{\lambda }$ is not affected by error), by the definition of $f$ it depends
on the the conditioning of $\left( A+\lambda I\right) ^{-1}A$. Denoting by $%
\kappa (\cdot )$ the condition number with respect to the Euclidean norm,
theoretically the best situation is attained defining $\lambda $ such that%
\begin{equation}
\kappa (A+\lambda I)=\kappa (\left( A+\lambda I\right) ^{-1}A),  \label{eq}
\end{equation}%
that is, the condition number of the preconditioner is equal to the
condition number of the preconditioned system. It is quite easy to prove
(see e.g. \cite{RDLS}) that in the SPD case taking $\lambda =\sqrt{\lambda
_{1}\lambda _{N}}$, where $\lambda _{1}$ and $\lambda _{N}$ are respectively
the smallest and the largest eigenvalue of $A$, we obtain $\kappa (A+\lambda
I)=\kappa (\left( A+\lambda I\right) ^{-1}A)=\sqrt{\kappa (A)}$.

The preconditioning effect of $A+\lambda I$ of course depends on the choice
of $\lambda $. By (\ref{eq}) it is necessary to find a compromise between
the preconditioning and the accuracy in the solution of the systems with $%
A+\lambda I$. In this sense formula (\ref{eq}), that theoretically
represents the optimal situation also implicitly states a lower bound for
the attainable accuracy. Indeed, many numerical experiments arising from the
discretization of Fredholm integral equation of the first kind, in which we
have examined the behavior of some classical Krylov methods such as the
GMRES and the CG preconditioned with $A+\lambda I$, have revealed that we
can substantially improve the rate of convergence (taking $\lambda \approx 1/%
\sqrt{\kappa (A)}$, see again \cite{RDLS} for a discussion) but we are not
able to improve the accuracy over a certain limit.

The ASP method can be extended to problems in which the exact right hand
side ${\mathbf{b}}$ is affected by noise. Anyway, since in presence of noise
a good approximation of the exact solution may be meaningless, we extend the
idea using the classical Tikhonov regularization. Moreover, many experiments
have shown that the ASP method generally produces poor results for problem
with noise.

We assume in particular to know only a perturbed right-hand side $\overline{{%
\mathbf{b}}}={\mathbf{b}}+{\mathbf{e}}_{{\mathbf{b}}}$, where ${\mathbf{e}}_{%
{\mathbf{b}}}$ is the perturbation. Given $\lambda >0$ and $H\in \mathbb{R}%
^{P\times N}$ such that $H^{T}H$ is non singular, for approximating the
solution of $A{\mathbf{x}}={\mathbf{b}}$ we solve the regularized system%
\begin{equation}
(A^{T}A+\lambda H^{T}H)\,{\mathbf{x}}_{\lambda }=A^{T}\,\overline{{\mathbf{b}%
}}.  \label{tt}
\end{equation}%
and then we approximate ${\mathbf{x}}$ by computing%
\begin{eqnarray}
\overline{{\mathbf{x}}} &=&\left( A^{T}A\right) ^{-1}(A^{T}A+\lambda H^{T}H)
\,{\mathbf{x}}_{\lambda }  \notag \\
&=&f(Q)\,{\mathbf{x}}_{\lambda }  \label{fq}
\end{eqnarray}%
where $f$ is defined by (\ref{fz}) and $Q=\left( H^{T}H\right)^{-1}\!\left(
A^{T}A\right) $. As before, for the computation of $f(Q){\mathbf{x}}%
_{\lambda }$ we use the standard Arnoldi method projecting the matrix$\ Q$\
onto the Krylov subspaces generated by $Q$ and ${\mathbf{x}}_{\lambda }$.
Now, at each step we have have to compute the vectors ${\mathbf{w}}_{j}=Q{%
\mathbf{v}}_{j}$, $j\geq 1$, with ${\mathbf{v}}_{1}={\mathbf{x}}_{\lambda
}/\left\Vert {\mathbf{x}}_{\lambda }\right\Vert $, that is, to solve the
systems
\begin{equation*}
\left( H^{T}H\right) {\mathbf{w}}_{j}=\left( A^{T}A\right) {\mathbf{v}}_{j}.
\end{equation*}%
This means that we actually do not need $Q$ explicitly. The algorithm is
almost identical to the one given for the ASP method, apart from the two
steps inserted in a box.

\begin{center}
\hrulefill \\[0pt]
\vspace{-0.1cm} \textbf{ATP Algorithm } \\[0pt]
\vspace{-0.2cm} \hrulefill \\[0pt]
\begin{tabbing}
NN \= XX \= XX \= XX \= XX \= \kill
{\bf Require} $A\in {\mathbb{R}}^{N\times N},\;\overline{{\mathbf{b}}}\in {\mathbb{R}}^{N}$, $ \lambda\in {\mathbb{R^{+}}}$ \\
{\bf Define} $f(z) = 1+\lambda z^{-1}$\\
\fbox{{\bf Solve} $(A^{T}A+\lambda H^{T}H){\mathbf{x}}_{\lambda }=A^{T}\overline{{\mathbf{b}}}$}\\
 ${\mathbf{v}}_{1} \leftarrow {\mathbf{x}}_{\lambda }/\Vert {\mathbf{x}}_{\lambda }\Vert $\\
{\bf for} $m=1,2,\ldots$ {\bf do}\\
\> \fbox{{\bf Solve} $\left( H^{T}H\right) {\mathbf{w}}_{m}=\left( A^{T}A\right) {\mathbf{v}}_{m}$} \\
\>   $h_{k,m} \leftarrow {\mathbf{v}}_{k}^{T}{\mathbf{w}}_{m}$\\
\>  $\widetilde{{\mathbf{v}}} \leftarrow {\mathbf{w}}_{m}-\sum\nolimits_{k=1}^{m}h_{k,m}{\mathbf{v}}_{k}$\\
\>  $h_{m+1,m} \leftarrow \Vert \widetilde{{\mathbf{v}}}\Vert $ \\
\>  ${\mathbf{v}}_{m+1} \leftarrow \widetilde{{\mathbf{v}}}%
/h_{m+1,m}$\\
\> {\bf Compute} $f(H_{m})$ by Schur-Parlett algorithm\\
\> ${\mathbf{x}}_{m} \leftarrow \Vert {\mathbf{x}}_{\lambda }\Vert V_{m}f(H_{m}){\mathbf{e}}_{1}$\\
{\bf end for}
\end{tabbing}
\hrulefill
\end{center}

This kind of approach is somehow related with the so-called iterated
Tikhonov regularization (see for instance \cite{HH} or \cite{Neu}), with the
important difference that now only one regularized system has to be solved.

\begin{remark}
\label{r1}It is worth noting that the matrix $Q$ is $H^{T}H$-symmetric, that
is, for each ${\mathbf{v}}, {\mathbf{w}}\in {\mathbb{R}}^{N}$%
\begin{equation*}
{\mathbf{v}}^{T}\left( H^{T}HQ\right) ^{T}{\mathbf{w}}={\mathbf{v}}%
^{T}\left( H^{T}HQ\right) {\mathbf{w}}= {\mathbf{v}}^{T}A^{T}A{\mathbf{w}}.
\end{equation*}%
Because of this property the ATP method can be symmetrized using the Lanczos
process based on this new metric. While this approach is promising because
computationally less expensive, some preliminary experiments have revealed
that it is also quite unstable and, in general, less accurate than the ATP
method. For this reason the analysis presented in the next sections does not
regard this symmetric variant, that we plan to consider in a future work.
\end{remark}

\section{Error analysis}

\label{sec3} In exact arithmetic the error of the ASP method is given by $%
E_{m}:={\mathbf{x}}-{\mathbf{x}}_{m}$ where ${\mathbf{x}}_{m}$ is defined by
(\ref{asp}). If we denote by $\Pi _{m-1}$ the vector space of polynomials of
degree at most $m-1$, it can be seen that
\begin{equation}
{\mathbf{x}}_{m}=p_{m-1}(A){\mathbf{x}}_{\lambda },  \label{pol}
\end{equation}%
where ${\mathbf{x}}_{\lambda }$ is the solution of (\ref{sysf1}) and $%
p_{m-1}\in \Pi _{m-1}$ interpolates, in the Hermite sense, the function $f$
at the eigenvalues of $H_{m}$, the so-called Ritz values. Exploiting the
interpolatory nature of the standard Arnoldi method, we notice, as pointed
out also in \cite{Eier}, that the error can be expressed in the form%
\begin{equation}
E_{m}=\left\Vert {\mathbf{x}}_{\lambda }\right\Vert g_{m}(A)q_{m}(A){\mathbf{%
v}}_{1},\quad {\mathbf{v}}_{1}={\mathbf{x}}_{\lambda }/\left\Vert {\mathbf{x}%
}_{\lambda }\right\Vert ,  \label{erapr}
\end{equation}%
where%
\begin{equation*}
q_{m}(z)=\det (zI-H_{m}),
\end{equation*}%
(see also \cite{MN1}), and
\begin{equation*}
g_{m}(z):=\frac{f(z)-p_{m-1}(z)}{\det (zI-H_{m})}.
\end{equation*}%
From (\ref{erapr}), a bound for $\left\Vert E_{m}\right\Vert $ can be
derived working with the field of values of $A$, defined as%
\begin{equation*}
F(A):=\left\{ \frac{{\mathbf{y}}^{H}A{\mathbf{y}}}{{\mathbf{y}}^{H}{\mathbf{y%
}}},{\mathbf{y}}\in \mathbb{C}^{N}\mathbf{\backslash }\left\{ 0\right\}
\right\} .
\end{equation*}%
Indeed, we can state the following result (see also the recent papers \cite%
{BR} and \cite{DM} for a background about the error analysis of the standard
Arnoldi method for matrix functions).

\begin{proposition}
\label{pg}Assume that $F(A)\subset \mathbb{C}^{+}$. Then%
\begin{equation}
\left\Vert E_{m}\right\Vert \leq K \; \frac{\lambda \left\Vert {\mathbf{x}}%
_{\lambda }\right\Vert}{a^{m+1}}\prod_{i=1}^{m}h_{i+1,i} ,  \label{errb}
\end{equation}%
where $a>0$ is the leftmost point of $F(A)$ and $2\leq K\leq 11.08$. In the
symmetric case we can take $K=1$.
\end{proposition}

\begin{proof}
From \cite{Cro}, we know that%
\begin{equation*}
\left\Vert g_{m}(A)\right\Vert \leq K\max_{z\in F(A)}\left\vert
g_{m}(z)\right\vert ,
\end{equation*}%
with $2\leq K\leq 11.08$, and hence by (\ref{erapr})
\begin{equation*}
\left\Vert E_{m}\right\Vert \leq K\left\Vert {\mathbf{x}}_{\lambda
}\right\Vert \max_{z\in F(A)}\left\vert g_{m}(z)\right\vert \left\Vert
q_{m}(A){\mathbf{v}}_{1}\right\Vert .
\end{equation*}%
Now $g_{m}(z)$ is a divided difference that can be bounded using the
Hermite-Genocchi formula (see e.g. \cite{Deb}), so that
\begin{eqnarray*}
\left\vert g_{m}(z)\right\vert &\leq &\frac{1}{m!}\max_{\xi \in co\left\{
z,\sigma (H_{m})\right\} }\left\vert \frac{d^{m}}{d\xi ^{m}}\left( 1+\frac{%
\lambda }{\xi }\right) \right\vert , \\
&\leq &\max_{\xi \in co\left\{ z,\sigma (H_{m})\right\} }\frac{\lambda }{%
\left\vert \xi \right\vert ^{m+1}}
\end{eqnarray*}%
where $co\left\{ z,\sigma (H_{m})\right\} $ denotes the convex hull of the
point set given by $z$ and $\sigma (H_{m})$. Since $\sigma (H_{m})\subset
F(H_{m})\subseteq F(A)$, by some well known properties of the Arnoldi
algorithm, and using the relation%
\begin{equation*}
\left\Vert q_{m}(A){\mathbf{v}}_{1}\right\Vert
=\prod\limits_{i=1}^{m}h_{i+1,i},
\end{equation*}%
that arises from (\ref{cla}) (see \cite{Morno}), the result follows.
\end{proof}

Since $a$, the leftmost point of $F(A)$, can be really small for the
problems we are dealing with, formula (\ref{errb}) can surely be considered
too pessimistic with respect to what happens in practice. However, the upper
bound given by (\ref{errb}) allows to derive some important information
about the behavior of the error. First of all, it states that the rate of
convergence is little influenced by the choice of $\lambda $, and this is
confirmed by the analysis given in Section \ref{sec4} and by the numerical
experiments. Secondly, it states that, independently of its magnitude, the
error decay is related with the rate of the decay of $\prod%
\nolimits_{i=1}^{m}h_{i+1,i}$. We need the following result (cf. \cite{Ne}
Theorem 5.8.10).

\begin{theorem}
\label{nev}Let $\sigma _{j}$ and $\lambda _{j}$, $j\geq 1$, be respectively
the singular values and the eigenvalues of an operator $A$. Assume that $%
\left\vert \lambda _{j}\right\vert \geq \left\vert \lambda _{j+1}\right\vert
$ and%
\begin{equation}
\sum_{j\geq 1}\sigma _{j}^{p}<\infty \text{ for a certain }0<p\leq 1\text{. }
\label{dec}
\end{equation}%
Let $s_{m}(z)=\prod\nolimits_{i=1}^{m}(z-\lambda _{i})$. Then%
\begin{equation*}
\left\Vert s_{m}(A)\right\Vert \leq \left( \frac{\eta \, e \, p}{m}\right)
^{m/p},
\end{equation*}%
where%
\begin{equation*}
\eta \leq \frac{1+p}{p}\sum_{j\geq 1}\sigma _{j}^{p}.
\end{equation*}
\end{theorem}

Of course, the hypothesis (\ref{dec}) is fulfilled by many problems arising
from the discretization of integral equations, in many cases with $p$ quite
small. Now, using the relation (\cite{Trefe} p. 269),
\begin{equation*}
\prod_{i=1}^{m}h_{i+1,i}\leq \left\Vert s_{m}(A){\mathbf{v}}\right\Vert
\end{equation*}%
that holds for each monic polynomial $s_{m}$ of exact degree $m$, we can say
that Theorem \ref{nev} reveals that for discrete ill-posed problems the rate
of decay of $\prod\nolimits_{i=1}^{m}h_{i+1,i}$ is superlinear and depends
on the $p$-summability of the singular values of $A$, i.e., on the degree of
ill-posedness of the problem (cf. \cite{Hof} Def. 2.42).\bigskip

In computer arithmetics, we need to assume that ${\mathbf{x}}_{\lambda }$,
solution of (\ref{sysf1}) is approximated by $\overline{{\mathbf{x}}}%
_{\lambda }$ with an accuracy depending on the choice of $\lambda $ and the
method used. In this way, the Arnoldi algorithm actually constructs the
Krylov subspaces $K_{m}(A,\overline{{\mathbf{x}}}_{\lambda })$. Hence the
error can be written as%
\begin{eqnarray}
\left\Vert \overline{E}_{m}\right\Vert &=&\left\Vert f(A){\mathbf{x}}%
_{\lambda }-\left\Vert \overline{{\mathbf{x}}}_{\lambda }\right\Vert
V_{m}f(H_{m}){\mathbf{e}}_{1}\right\Vert \leq  \notag \\
&&\left\Vert f(A)\overline{{\mathbf{x}}}_{\lambda }-\left\Vert \overline{{%
\mathbf{x}}}_{\lambda }\right\Vert V_{m}f(H_{m}){\mathbf{e}}_{1}\right\Vert
+\left\Vert f(A)\left( {\mathbf{x}}_{\lambda }-\overline{{\mathbf{x}}}%
_{\lambda }\right) \right\Vert .  \label{2p}
\end{eqnarray}%
The above formula expresses the error in two terms, one depending on the
accuracy of the Arnoldi method for matrix functions and one on the accuracy
in the computation of ${\mathbf{x}}_{\lambda }$. Roughly speaking we can
state that for small values of $\lambda $, $f(A)\approx I$ (cf. (\ref{fa}))
and we have that $\left\Vert \overline{E}_{m}\right\Vert \approx \left\Vert {%
\mathbf{x}}_{\lambda }-\overline{{\mathbf{x}}}_{\lambda }\right\Vert $. This
means that the method is not able to improve the accuracy provided by the
solution of the initial system. For large $\lambda $ we have that ${\mathbf{x%
}}_{\lambda }\approx \overline{{\mathbf{x}}}_{\lambda }$ because the system (%
\ref{sysf1}) is well conditioned, but even assuming that $\left\Vert
f(A)\left( {\mathbf{x}}_{\lambda }-\overline{{\mathbf{x}}}_{\lambda }\right)
\right\Vert \approx 0$ that in principle may happen even if $\left\Vert
f(A)\right\Vert $ is large, we have another lower bound due the ill
conditioning of $f(A)=A^{-1}\left( A+\lambda I\right) $ since now $A+\lambda
I$ has a poor effect as preconditioner.

Regarding the optimal choice of $\lambda $ we can make the following
consideration. Unless the re-orthogonalization or the Householder
implementation is adopted, the Arnoldi method typically stagnates around the
best approximation ${\mathbf{x}}_{\overline{m}}$ because of the loss of
orthogonality of the Krylov basis. Therefore let $c(\lambda )$ be such that%
\begin{equation*}
\left\Vert f(A)\overline{{\mathbf{x}}}_{\lambda }-\left\Vert \overline{{%
\mathbf{x}}}_{\lambda }\right\Vert V_{m}f(H_{m}){\mathbf{e}}_{1}\right\Vert
\rightarrow c(\lambda )\text{ as }m\rightarrow N.
\end{equation*}%
Then by (\ref{2p}) the optimal value of $\lambda $ depends on the method
used to compute $\overline{{\mathbf{x}}}_{\lambda }$ and is given by%
\begin{equation}
\lambda _{\mathrm{opt}}=\arg \min_{\lambda >0}\left( c(\lambda )+\left\Vert
f(A)\left( {\mathbf{x}}_{\lambda }-\overline{{\mathbf{x}}}_{\lambda }\right)
\right\Vert \right) .  \label{lo}
\end{equation}%
Of course the above formula is interesting only by a theoretical point of
view. In practice, as mentioned in the introduction, one could try to
compare the conditioning of $A+\lambda I$ and $f(A)$, by approximating the
solution of%
\begin{equation}
\kappa (A+\lambda I)=\kappa (\left( A+\lambda I\right) ^{-1}A).  \label{cc}
\end{equation}%
with respect to $\lambda $. However, since the computation of ${\mathbf{x}}%
_{\lambda }$ comes first, it is suitable to take $\lambda $ a bit larger
than the solution of (\ref{cc}). Note that generally such solution can be
approximated by $\lambda =1/\kappa (A)$.\bigskip

For the ATP method the analysis is almost identical since the error is given
by%
\begin{equation*}
\overline{E}_{m}:=f(Q){\mathbf{x}}_{\lambda }-\left\Vert \overline{{\mathbf{x%
}}}_{\lambda }\right\Vert V_{m}f(H_{m}){\mathbf{e}}_{1},
\end{equation*}%
where $(A^{T}A+\lambda H^{T}H){\mathbf{x}}_{\lambda }=A^{T}{\mathbf{b}}$, $%
(A^{T}A+\lambda H^{T}H)\overline{{\mathbf{x}}}_{\lambda }=A^{T}\overline{{%
\mathbf{b}}}$, and $Q=(H^TH)^{-1}(A^TA)$. Hence, as before we have%
\begin{equation*}
\left\Vert \overline{E}_{m}\right\Vert \leq \left\Vert f(Q)\overline{{%
\mathbf{x}}}_{\lambda }-p_{m-1}(Q)\overline{{\mathbf{x}}}_{\lambda
}\right\Vert +\left\Vert f(Q)\left( {\mathbf{x}}_{\lambda }-\overline{{%
\mathbf{x}}}_{\lambda }\right) \right\Vert ,
\end{equation*}%
where $p_{m-1}$ is again defined by (\ref{pol}). This expression is
important since it states that theoretically we may take $\lambda $ very
large, thus oversmoothing, in order to reduce the effect of noise and then
leaving to the Arnoldi algorithm the task of recovering the solution.
Unfortunately, the main problem is that, as before, $f(Q)$ may be
ill-conditioned for $\lambda $ large. Henceforth, even in this case we
should find a compromise for the selection of a suitable value of $\lambda $%
, but contrary to the ASP method for noise-free problems it is difficult to
design a theoretical strategy. Indeed everything depends on the problem and
on the operator $H$. In most cases the noise on the right-hand side produces
an increment of the high-frequency components of ${\mathbf{b}}$, that are
emphasized on the solution by the nature of the problem. For this reason $H$
is generally taken as a high-pass filter, as for instance a derivative
operator, and the solution of (\ref{mn}) can be interpreted as a numerical
approximation via penalization of the constrained minimization problem%
\begin{equation*}
\min_{\left\Vert H{\mathbf{x}}\right\Vert =0}\left\Vert A{\mathbf{x}}-%
\overline{{\mathbf{b}}}\right\Vert
\end{equation*}%
While in standard constrained minimization one approximates the solution
taking $\lambda $ very large (theoretically $\lambda \rightarrow \infty $),
in our case $H$ is hardly able to detect efficiently the effect of noise on
the numerical solution so that one is forced to adopt some heuristic
criterium such as the L-curve analysis. In general terms we can say that if
the solution is smooth and involves only low frequencies then a high-pass
filter should lead to a good approximation taking $\lambda $ ``large''. On
the other side if the solution involves itself high-frequencies as in the
case of discontinuities, then it is better to undersmooth the problem so
reducing the effect of the filter. We have made these considerations just to
point out that a general theoretical indication on the choice of $\lambda $
is not possible dealing with problems affected by error. What we can do is
to derive methods able to reduce the dependence on this choice, and the ATP
method seems to have some chances in this direction.

\section{Filter factors}

\label{sec4} In order to understand the action of the second phase of the
methods, i.e., the matrix function evaluation applied to the regularized
solution (cf. (\ref{fa}) and (\ref{fq})), below we investigate the
corresponding filter factors.

Assuming for simplicity that $A$ is diagonalizable, that is, $A=XDX^{-1}$
where $D=\mathrm{diag}(\lambda _{1},...,\lambda _{N})$, for the ASP method
we have%
\begin{equation*}
{\mathbf{x}}_{\lambda }=\sum_{i=1}^{N}\frac{\lambda _{i}}{\lambda
_{i}+\lambda }\frac{\left( X^{-1}{\mathbf{b}}\right) _{i}}{\lambda _{i}}\, {%
\mathbf{x}}_{i},
\end{equation*}%
where ${\mathbf{x}}_{i}$ is the eigenvector associated with $\lambda _{i}$,
and $\left( \cdot \right) _{i}$ denotes the $i$-th component of a vector.
After the first phase, the filter factors are thus $g_{i}=\lambda_{i}\left(
\lambda _{i}+\lambda \right)^{-1}$. Since from (\ref{pol}), we have ${%
\mathbf{x}}_{m}=p_{m-1}(A){\mathbf{x}}_{\lambda },$ where $p_{m-1}$
interpolates the function $f$ at the eigenvalues of $H_{m}$, we immediately
obtain%
\begin{equation*}
{\mathbf{x}}_{m}=\sum_{i=1}^{N}\frac{\lambda _{i}p_{m-1}(\lambda _{i})}{%
\lambda _{i}+\lambda }\frac{\left( X^{-1}{\mathbf{b}}\right) _{i}}{\lambda
_{i}}\, {\mathbf{x}}_{i}.
\end{equation*}%
Therefore, at the $m$-th step of the ASP method the filter factors are given
by%
\begin{equation*}
f_{i}^{(m)}=\frac{\lambda _{i}p_{m-1}(\lambda _{i})}{\lambda _{i}+\lambda }%
,\quad i=1,...,N\text{.}
\end{equation*}%
Let us compare, with an example, the behavior of the filter factors.
Similarly to what was made in \cite{PCH}, we consider the problem GRAVITY
taken from the Hansen's \texttt{Regularization Tools} \cite{H1, H2}, with
dimension $N=12$. In Figure \ref{figure0}, the filter factors $g_{i}$ and $%
f_{i}^{(m)}$, for $m=4,6,8,10$ are plotted. As regularization parameter we
have chosen $\lambda =1/\sqrt{\kappa (A)}$. Since the problem is SPD, for
more clarity in the pictures, the eigenvalues $\lambda _{i}$ have been
sorted in decreasing order.

\begin{figure}[h]
\centerline{\includegraphics[width=12cm]{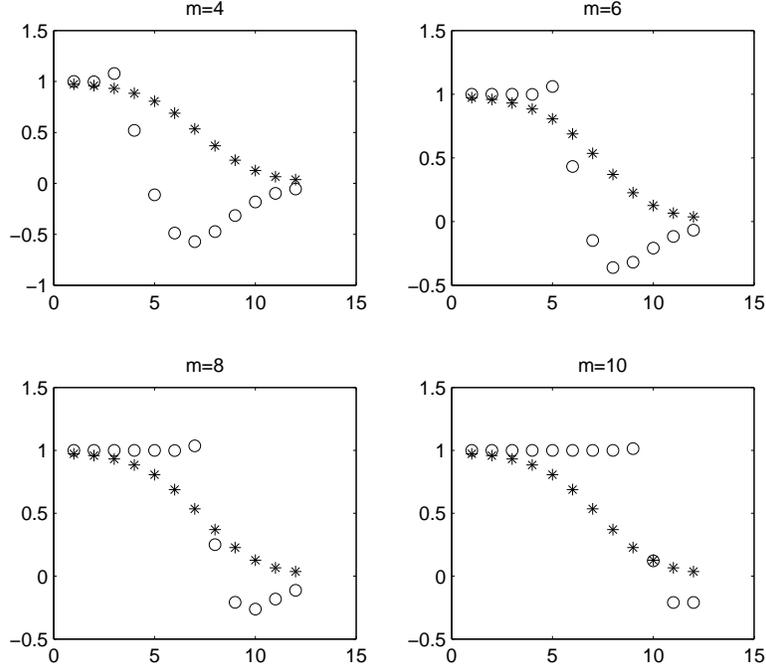}} \vspace{-0.4cm}
\caption{Filter factors $g_{i}$ (asterisk) and $f_{i}^{(m)}$ (circle) with $%
m=4,6,8,10$, for GRAVITY(12).}
\label{figure0}
\end{figure}

While the problem is rather simple the pictures clearly represent the action
of the Arnoldi (Lanczos in this case) steps. Since the Arnoldi (Lanczos)
algorithm initially picks up the largest eigenvalues, it automatically
corrects the filters corresponding to the low-middle frequencies ($%
g_{i}\rightarrow f_{i}^{(m)}\approx 1$), keep damping the highest ones. The
second phase thus performs a correction, but the properties of the Arnoldi
algorithm guarantees that the method can still be interpreted as a
regularizing approach.

For a better explanation of Figure \ref{figure0}, let us assume for
simplicity that the Ritz values $r_{j}$, $j=1,...,m$, are distinct (as in
the example), so that we can write
\begin{equation*}
p_{m-1}(\lambda _{i})=\sum_{j=1}^{m}l_{j}(\lambda _{i})f(r_{j}),
\end{equation*}%
where $l_{j}$, $j=1,...,m$ are the Lagrange polynomials. Hence we obtain%
\begin{equation*}
f_{i}^{(m)}=\sum_{j=1}^{m}l_{j}(\lambda _{i})\frac{\lambda _{i}}{r_{j}}\frac{%
r_{j}+\lambda }{\lambda _{i}+\lambda },\quad i=1,...,N\text{.}
\end{equation*}%
Since the Arnoldi algorithm ensures that $r_{j}\approx \lambda _{j}$ for $%
j=1,...,m$ we have $f_{i}^{(m)}\approx 1$ for $i\leq m$. For $i>m$ and when $%
\lambda _{i}\approx 0$ we have that%
\begin{equation*}
f_{i}^{(m)}\approx p_{m-1}(0)\frac{\lambda _{i}}{\lambda _{i}+\lambda },
\end{equation*}
so that the filters are close to the ones of the uncorrected scheme. Of
course, numerically, the problems start to appear when the Arnoldi algorithm
fails to provide good approximations of the eigenvalues of $A$, but it is
important to observe that, at least in exact arithmetics, the choice of $%
\lambda $ only influences the high frequencies. For this reason, at least
for the ASP method, this choice is more related to the conditioning of the
subproblems (cf. Section \ref{sec3}).

The filter factor analysis just presented remains valid also for the ATP
method. Taking $H=I$ in (\ref{tt}) and using the SVD decomposition we easily
find that the filter factors are now given by%
\begin{equation*}
f_{i}^{(m)}=\frac{\sigma _{i}^{2}\, p_{m-1}(\sigma _{i}^{2})}{\sigma
_{i}^{2}+\lambda }
\end{equation*}%
and hence our considerations for the ASP method remains true also for this
case. Of course for $H\neq I$ we just need to consider the GSVD. For
problems with noise, the choice of $\lambda $ is of great importance. Anyway
we have just seen that the correction phase allows to reproduce the low
frequencies independently of this choice. In this sense, in practice we can
take $\lambda $ even very large in order to reduce as much as possible the
influence of noise.

\section{Numerical experiments}

\label{sec5} This section is devoted to the numerical experiments obtained
on a single processor computer Intel Core Duo T5800 with Matlab 7.9. Our
goal is to prove numerically what we consider the valuable properties of the
ASP and the ATP methods, that is, accuracy and speed comparable with the
most effective iterative solvers, stability, and reduced dependence on the
parameter $\lambda $. For the experiments we consider problems taken from
the \texttt{Regularization Tools Matlab} package by Hansen \cite{H1,H2}. Our
comparison method is the {Matlab} version of the GMRES, that is implemented
with the Householder algorithm that guarantees the orthogonality of the
Krylov basis at the machine precision. For the problems here considered the
GMRES method has shown to be the most accurate, if compared to other well
known methods that we can found in the literature. Since it is also quite
unstable, it is generally implemented together with the discrepancy
principle as stopping criterium (where it is possible of course), but not
always with good results. We point out that the modified Gram-Schmidt
version of the GMRES has also been considered in the experiments (even if
not reported); this version is stable, but unfortunately the attainable
accuracy loses one or even two order of magnitude with respect to the
version implemented by {Matlab}. Other methods such as the CGLS and LSQR are
widely inferior for the problems here considered.

\begin{figure}[h]
\centerline{\includegraphics[width=12cm]{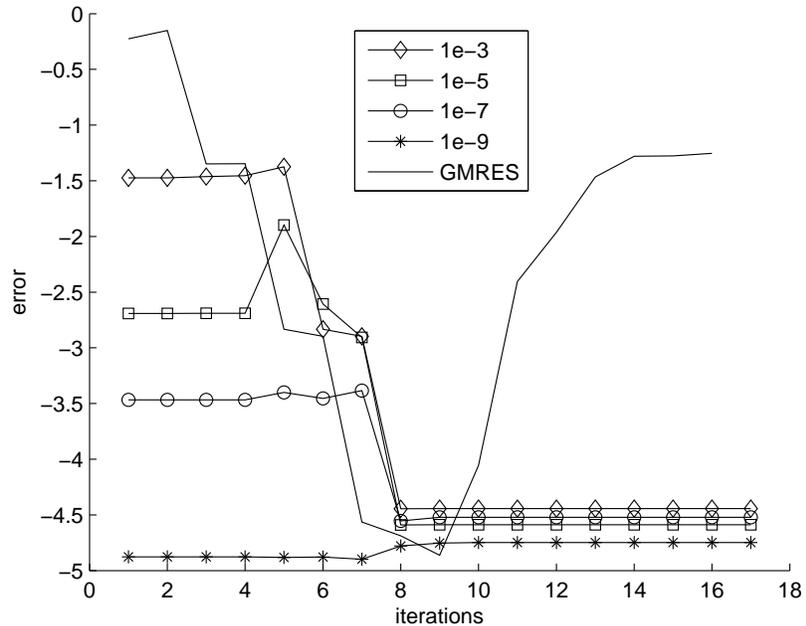}} \vspace{-0.4cm}
\caption{Error behavior of the GMRES and the ASP method with $\protect%
\lambda =10^{-3}, 10^{-5}, 10^{-7}, 10^{-9}$, for noise-free BAART(240).}
\label{figure1}
\end{figure}

In all experiments the Arnoldi algorithm for the ASP and the ATP methods, as
said in Section \ref{sec2}, is implemented with the modified Gram-Schmidt
orthogonalization, and the initial linear system is solved with the LU or
the Cholesky factorization.

\begin{table}[h]
\begin{center}
\begin{tabular}{lccc}
\hline
& error & residual & $\lambda $ \\ \hline
ASP & \multicolumn{1}{l}{$3.58 \times 10^{-5} (8)$} & \multicolumn{1}{l}{$%
1.89\times 10^{-12}$} & \multicolumn{1}{l}{$10^{-3}$} \\
& \multicolumn{1}{l}{$2.57\times 10^{-5} (8)$} & \multicolumn{1}{l}{$%
3.86\times 10^{-13}$} & \multicolumn{1}{l}{$10^{-5}$} \\
& \multicolumn{1}{l}{$2.78\times 10^{-5} (8)$} & \multicolumn{1}{l}{$%
4.79\times 10^{-14}$} & \multicolumn{1}{l}{$10^{-7}$} \\
& \multicolumn{1}{l}{$1.26\times 10^{-5} (7)$} & \multicolumn{1}{l}{$%
1.94\times 10^{-12}$} & \multicolumn{1}{l}{$10^{-9}$} \\ \hline
GMRES & \multicolumn{1}{l}{$1.37\times 10^{-5} (9)$} & \multicolumn{1}{l}{$%
2.21\times 10^{-15}$} & \multicolumn{1}{l}{}%
\end{tabular}%
\end{center}
\par
\vspace{-0.4cm}
\caption{Results for BAART(240) in the noise-free case. }
\label{table1}
\end{table}

\begin{figure}[h]
\centerline{\includegraphics[width=12cm]{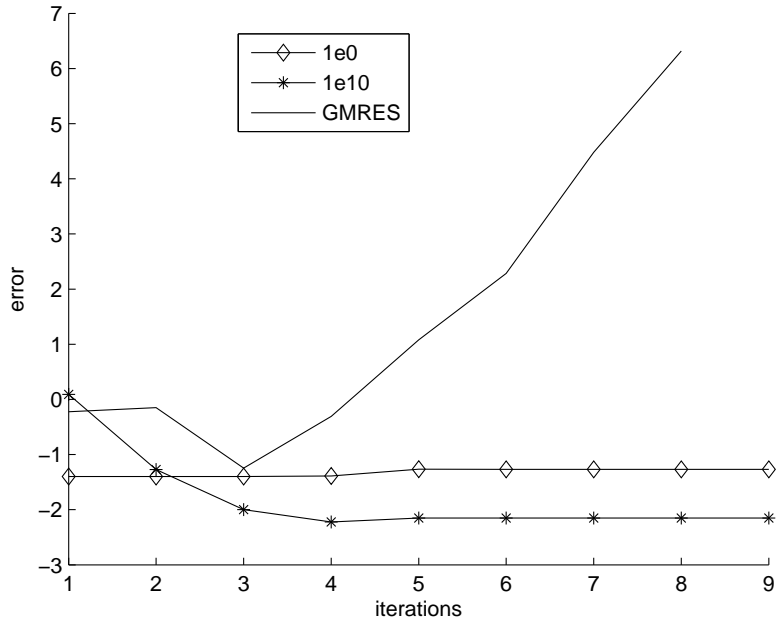}} \vspace{-0.4cm}
\caption{Error behavior of the GMRES and the ATP method with $\protect%
\lambda =1$ and $\protect\lambda =10^{10}$ for BAART(240) with Gaussian
noise.}
\label{figure2}
\end{figure}

\begin{table}[h]
\begin{center}
\begin{tabular}{lccc}
\hline
& error & residual & $\lambda $ \\ \hline
ATP & \multicolumn{1}{l}{$4.00 \times 10^{-2} (2)$} & \multicolumn{1}{l}{$%
2.70 \times 10^{-4}$} & \multicolumn{1}{l}{$1$} \\
& \multicolumn{1}{l}{$6.01 \times 10^{-3} (4)$} & \multicolumn{1}{l}{$2.17
\times 10^{-4}$} & \multicolumn{1}{l}{$10^{10}$} \\ \hline
GMRES & \multicolumn{1}{l}{$5.66 \times 10^{-2} (3)$} & \multicolumn{1}{l}{$%
2.16 \times 10^{-4}$} & \multicolumn{1}{l}{}%
\end{tabular}%
\end{center}
\caption{Results for BAART(240) with Gaussian noise. }
\label{table2}
\end{table}

As first test problem we consider BAART(240) (in parenthesis, as usual, we
indicate the dimension $N $). The estimated condition number of the
corresponding matrix $A$ is around $10^{20}$. We first consider the
noise-free case comparing the behavior of the ASP method with GMRES, taking
different values of the parameter $\lambda $. Looking at Figure \ref{figure1}
we can observe that even considering a wide range of values for $\lambda $,
contrary to GMRES the ASP method does not suffer from semi-convergence, that
is, the error always stabilizes around the minimum. The attainable accuracy
is always quite close to the one of GMRES. The number of iterations
necessary to achieve the minimum accuracy is almost always the same, as
expected from Proposition \ref{pg} and it depends on the spectral properties
of the operator, that is, on the fast decay of $\prod%
\nolimits_{i=1}^{m}h_{i+1,i}$ (cf. Theorem \ref{nev}).

Another important observation can be made looking at the error curve
corresponding to the choice of $\lambda =10^{-9}$ (line with asterisks).
Since this curve is almost flat we argue that this value of $\lambda $ is
probably very close to the value $\lambda _{\mathrm{opt}}$ defined by (\ref%
{lo}), that seeks for a compromise between the accuracies in the solutions
of the initial linear system and in the computation of the matrix function.
In Table \ref{table1} the minimal errors (with the iteration numbers in
parenthesis) and the corresponding residuals are reported.

Now we consider the same problem with right-hand side affected by noise. We
try to solve $A{\mathbf{x}}={\mathbf{b}}$ working with an inexact right-hand
side $\overline{{\mathbf{b}}}={\mathbf{b}}+{\mathbf{e}}_{{\mathbf{b}}}$
where ${\mathbf{e}}_{{\mathbf{b}}}$ is a noise vector of the type%
\begin{equation}
{\mathbf{e}}_{{\mathbf{b}}}=\frac{\delta \left\Vert {\mathbf{b}}\right\Vert
}{\sqrt{N}}\; {\mathbf{u}},  \label{nv}
\end{equation}%
where we define $\delta =10^{-3}$ as the relative noise level, and ${\mathbf{%
u}}=\mathtt{randn(N,1)}$, that in {Matlab} notation is a vector of $N$
random components with normal distribution with mean $0$ and standard
deviation $1$. For the ATP method, we define $H$ as the discrete second
derivative operator, that is,
\begin{equation*}
H=\left(
\begin{array}{cccc}
2 & -1 &  &  \\
-1 & \ddots & \ddots &  \\
& \ddots & \ddots & -1 \\
&  & -1 & 2%
\end{array}%
\right) \in \mathbb{R}^{N\times N},
\end{equation*}%
and we choose $\lambda =1$ and $\lambda =10^{10}$. The comparison is made
again with the GMRES. The error curves are plotted in Figure \ref{figure2}.
For $\lambda =1$ the method does not provide a substantial improvement to
the first iteration that correspond to the Cholesky solution of the Tikhonov
system. Probably this is due to the fact that $\lambda =1$ is close to the
value attainable with the L-curve analysis. Anyway it is important to notice
that the method does not deteriorate that approximation during the
iteration. For $\lambda =10^{10}$ we have an effective and stable
improvement with a good accuracy if compared with the one of the GMRES. In
order to avoid confusion in the pictures we only consider these two values,
since in the internal range the curves are similar, showing the robustness
of the method with respect to the choice of the parameter $\lambda $. The
results are reported in Table \ref{table2}.

For a fair comparison between the ASP method and GMRES we also consider the
preconditioned version of this code that we denote by PGMRES with the same
preconditioner used by the ASP method, that is, $A+\lambda I$. Working again
with BAART(240) with exact right-hand side, in Figure \ref{figure3} we plot
the error curves with respect to the computational cost. While a flops
counter is no longer available in Matlab, it is quite easy to derive these
numbers knowing the algorithms. The non-vectorial operations are neglected.
For both methods the systems with $A+\lambda I$\ are solved by means of the
LU factorization, computed only once at the beginning. Of course, each
PGMRES iteration is more expensive since it requires the solution of a
system with $A+\lambda I$.

\begin{figure}[h]
\centerline{\includegraphics[width=16cm]{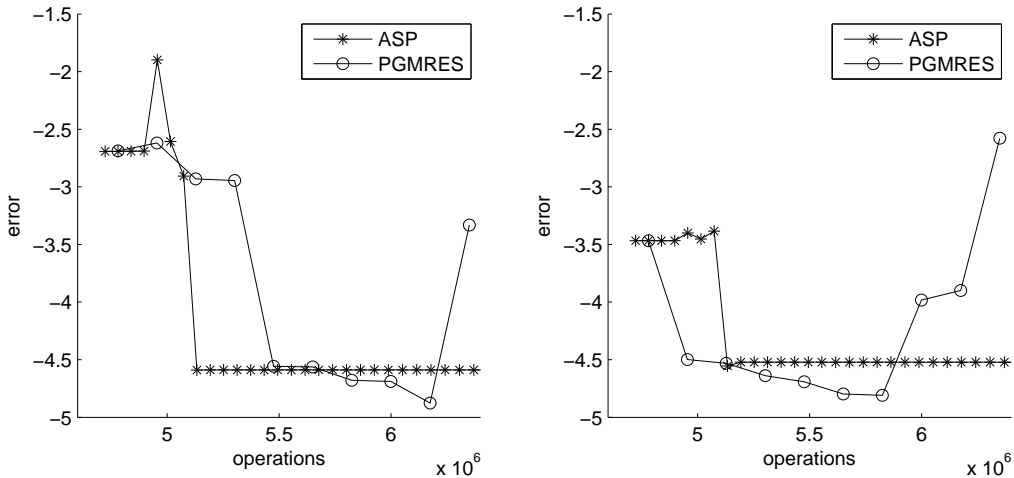}}
\caption{Error behavior of the preconditioned GMRES and the ASP method for
BAART(240) with $\protect\lambda =10^{-5}$ (left) and $10^{-7}$ (right).}
\label{figure3}
\end{figure}

\vspace{0.4cm}
\begin{table}[h]
\begin{center}
\begin{tabular}{lccc}
\hline
& error & residual & $\lambda $ \\ \hline
ASP & \multicolumn{1}{l}{$2.57 \times 10^{-5} (8)$} & \multicolumn{1}{l}{$%
3.86 \times 10^{-13}$} & \multicolumn{1}{l}{$10^{-5}$} \\
PGMRES & \multicolumn{1}{l}{$1.33 \times 10^{-5} (9)$} & \multicolumn{1}{l}{$%
3.92 \times 10^{-15}$} & \multicolumn{1}{l}{$10^{-5}$} \\ \hline
ASP & \multicolumn{1}{l}{$2.78 \times 10^{-5} (8)$} & \multicolumn{1}{l}{$%
4.79 \times 10^{-14}$} & \multicolumn{1}{l}{$10^{-7}$} \\
PGMRES & \multicolumn{1}{l}{$1.55 \times 10^{-5} (7)$} & \multicolumn{1}{l}{$%
3.43 \times 10^{-13}$} & \multicolumn{1}{l}{$10^{-7}$}%
\end{tabular}%
\end{center}
\par
\vspace{-0.4cm}
\caption{Comparison between the ASP method and the PGMRES for $\protect%
\lambda =10^{-5},10^{-7}$. }
\label{table3}
\end{table}
The results reported in Figure \ref{figure3} reveal that the ASP is still
competitive with the PGMRES in terms of accuracy and computational cost. For
this example the PGMRES is a bit faster than GMRES (cf. Figure \ref{figure1}%
) since the error curve is steeper at the beginning, but it remains
unstable. Comparing also the results of these examples (Table \ref{table3})
with the ones reported in Table \ref{table1}, we also observe a very little
improvement in terms of accuracy.

\begin{figure}[h]
\centerline{\includegraphics[width=13cm]{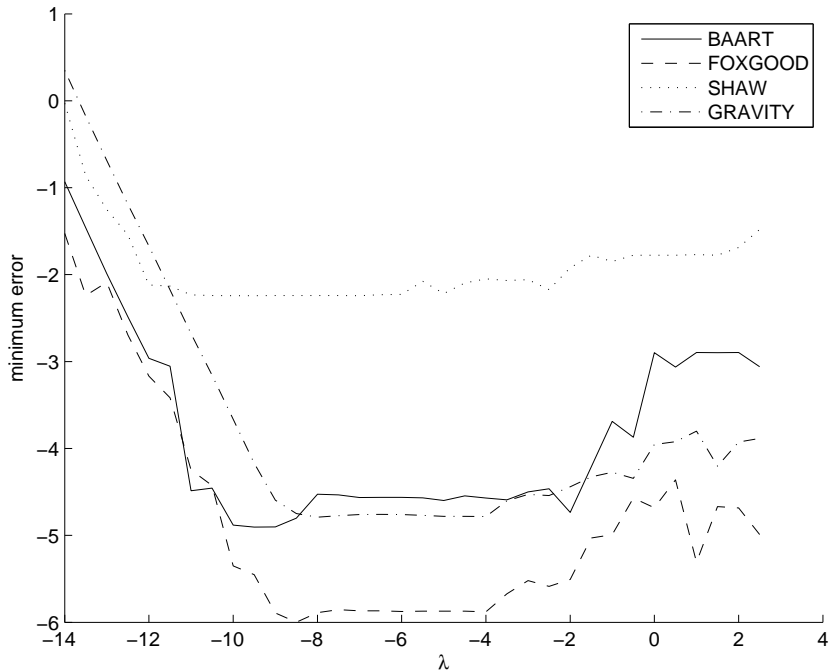}} \vspace{-0.4cm}
\caption{Maximum attainable accuracy with respect to the choice of $\protect%
\lambda$. The dimension of each problem is $N=160$.}
\label{figure4}
\end{figure}

In a final example we want show the behavior of the methods in four
classical problems (BAART, FOXGOOD, SHAW and GRAVITY), with $N=160$,
changing the value of the parameter $\lambda $. Figure \ref{figure4} is
representative of what happen in general for the ASP method with exact
right-hand side, that is, as expected, the attainable accuracy is generally
poor for small values of $\lambda $ (the initial system is badly solved) and
for large values of $\lambda $ (the preconditioning effort is poor). In any
case it is really important to observe that the maximum accuracy can be
obtained without much differences for a relatively large window of values
for $\lambda $, since the curves exhibit a plateau around the minimum.
Indicatively, we may say that the maximum accuracy can be achieved taking $%
\lambda $ in a range between $1/\sqrt{\kappa (A)}$ and $1/\sqrt[4]{\kappa (A)%
}$. The importance of this behavior is not negligible because it means that
having an estimate of the conditioning of $A$ allows to skip any
pre-processing techniques to estimate the optimal value of $\lambda $.

\begin{figure}[h]
\centerline{\includegraphics[width=13cm]{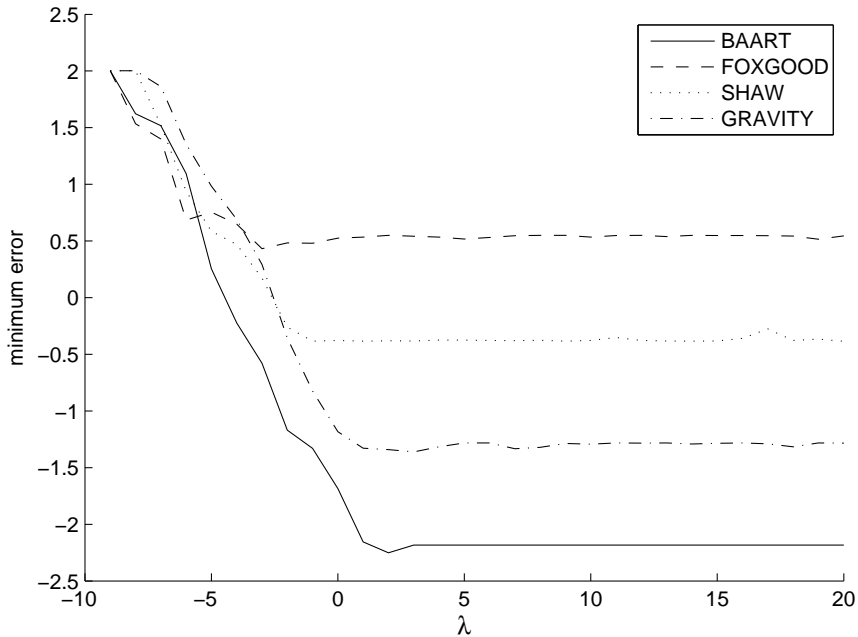}} \vspace{-0.4cm}
\caption{Maximum attainable accuracy with respect to the choice of lambda
with right-hand side affected by noise. The dimension of each problem is $%
N=160$.}
\label{figure5}
\end{figure}

Assume now to work with a right-hand side affected by noise, $\overline{{%
\mathbf{b}}}={\mathbf{b}}+{\mathbf{e}}_{{\mathbf{b}}}$, where ${\mathbf{e}}_{%
{\mathbf{b}}}$ is a defined by (\ref{nv}) with noise level $\delta =10^{-3}$%
. Looking at Figure \ref{figure5}, we can observe that with respect to the
noise-free case we do not even have the problem of oversmoothing taking $%
\lambda $ too large, at least for the example considered. We argue that the
bottleneck, for what concerns the accuracy, is represented by the effect of
noise. In general, increasing the value of $\lambda $ leads to a slight
increase of the number of iterations. These considerations leads us to state
a general strategy for an automatic parameter-choice implementation of the
method.

\begin{enumerate}
\item define $\lambda $ relatively ``large'', for instance even much larger
than the point of maximum curvature of the L-curve;

\item use any parameter-choice method for $m$ to define the stopping rule
(as for instance the discrepancy principle where possible), allowing some
more iterations to avoid oversmoothing ($m$ too small, cf. Figure \ref%
{figure2}).
\end{enumerate}

Concluding we may say that for the ATP method of course there exists an
optimal value of $\lambda $, say $\lambda _{\mathrm{opt}}$, close to the
corners of the L-shaped curves of Figure \ref{figure5}, and a corresponding $%
m_{\mathrm{opt}}$, that is, the minimum number of iterations to achieve the
optimal accuracy. Anyway, our experiments reveal that working with $\lambda
>\lambda _{\mathrm{opt}}$ and $m>m_{\mathrm{opt}}$, we do not have a
sensible loss of accuracy nor a remarkable increase of computational cost.

\section{An example of image restoration}

\label{sec6} In this section we consider a problem of image restoration. The
example is a 2D image deblurring problem which consists of recovering the
original $n\times n$ image from a blurred and noisy image. The original
image is denoted by $X$ and it consists of $n\times n$ grayscale pixel
values. Let ${\mathbf{x}}=vec\left( X\right) \in \mathbb{R}^{N}$, $N=n^{2}$,
be the vector whose entries are the pixel values of the image $X$. Let
moreover $A\in \mathbb{R}^{N\times N}$ be the matrix representing the
blurring operator, coming from the discretization of the Point Spread
Function (PSF). The vector ${\mathbf{b}}=A{\mathbf{x}}$ represents the
associated blurred and noise-free image. We generate a blurred and noisy
image $\overline{{\mathbf{b}}}={\mathbf{b}}+{\mathbf{e}}_{{\mathbf{b}}}$,
where ${\mathbf{e}}_{{\mathbf{b}}}$ is a noise vector defined by (\ref{nv})$%
\ $with $\delta =10^{-3}$.

The matrix $A$ is a symmetric Toeplitz matrix given by
\begin{equation*}
A=(2\pi \sigma ^{2})^{-1}T\otimes T
\end{equation*}%
where $T$ is a $n\times n$ symmetric banded Toeplitz matrix whose first row
is a vector ${\mathbf{v}}$ whose element are%
\begin{equation*}
v_{j}:=\left\{
\begin{array}{cl}
\displaystyle\frac{e^{-(j-1)^{2}}}{2\sigma ^{2}} & \text{for }j=1,...,q \\
0 & \text{for }j=q+1,...,n%
\end{array}%
\right.
\end{equation*}%
The parameter $q$ is the half-bandwidth of the matrix $T$, and the parameter
$\sigma $ controls the width of the underlying Gaussian point spread function

\begin{equation*}
h(x,y)=\frac{1}{2\pi \sigma ^{2}}\exp \left( -\frac{x^{2}+y^{2}}{2\sigma ^{2}%
}\right) ,
\end{equation*}%
which models the degradation of the image. Thus, a larger $\sigma $ implies
a wider Gaussian and thus a more ill-posed problem. For our experiments $X$
is a $100\times 100$ subimage of the image \texttt{coins.png} from Matlab's
Image Processing Toolbox, shown as the first image in Figure \ref{figure6}.
We define $q=6$ and $\sigma =1.5$, so that the condition number of $A$ is
around $10^{10}$. We report the results of our image restoration using two
different\ regularization operators. In particular we consider the matrix%
\begin{equation*}
H_{1,2}=\left(
\begin{array}{c}
I\otimes H_{1} \\
H_{1}\otimes I%
\end{array}%
\right) ,\text{ where }H_{1}=\left(
\begin{array}{cccc}
1 & -1 &  &  \\
& \ddots & \ddots &  \\
&  & 1 & -1 \\
&  &  & 1%
\end{array}%
\right) \in \mathbb{R}^{n\times n},
\end{equation*}%
taken from \cite{KHE} (with a slight modification in order that $%
H_{1,2}^{T}H_{1,2}$ is nonsingular), and the matrix $H_{2,2}$ defined as the
discretization of the two-dimensional Laplace operator with zero-Dirichlet
boundary conditions, that is,%
\begin{equation*}
H_{2,2}=\left(
\begin{array}{ccccc}
4 & -1 &  & -1 &  \\
-1 & 4 & -1 &  & -1 \\
& \ddots & \ddots & \ddots &  \\
-1 &  & -1 & 4 & -1 \\
& -1 &  & -1 & 4%
\end{array}%
\right) \in \mathbb{R}^{N\times N}.
\end{equation*}

\begin{figure}[h]
\centerline{\includegraphics[width=15cm]{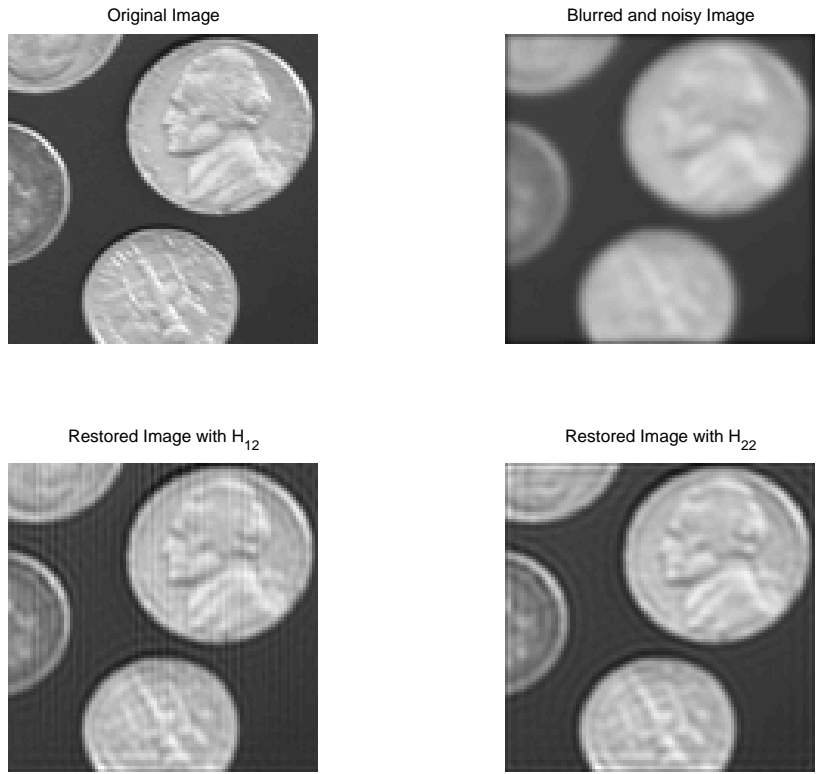}} \vspace{-1.2cm}
\caption{Image restoration with the ATP method using $H_{1,2}$, $H_{2,2}$
and $\protect\lambda =100$.}
\label{figure6}
\end{figure}

Figure \ref{figure6} shows that the ATP method can be fruitfully used also
for these kind of problems. Due to the well marked edges, the original image
involves high-frequencies so that the restoration by means of the standard
derivative operators is intrinsically difficult, because they are high-pass
filters.

Table \ref{table4} shows that also for this kind of problems the attainable
accuracy is weakly influenced by the choice of $\lambda $.

\vspace{0.4cm}
\begin{table}[h]
\begin{center}
\begin{tabular}{lcccc}
\hline
& $1$ & $10^{2}$ & $10^{4}$ & $10^{6}$ \\ \hline
$H_{1,2}$ & \multicolumn{1}{l}{0.060} & \multicolumn{1}{l}{0.060} &
\multicolumn{1}{l}{0.062} & \multicolumn{1}{l}{0.059} \\
$H_{2,2}$ & \multicolumn{1}{l}{0.061} & \multicolumn{1}{l}{0.064} &
\multicolumn{1}{l}{0.069} & \multicolumn{1}{l}{0.075}%
\end{tabular}%
\end{center}
\par
\vspace{-0.4cm}
\caption{Attainable accuracy (Euclidean norm of the error) for the image
restoration with $H_{1,2}$ and $H_{2,2}$ using different values of $\protect%
\lambda$. }
\label{table4}
\end{table}

\section{Conclusions}

\label{sec7} In this paper we have presented a new approach for the solution
of discrete ill-posed problems. The basic idea is to solve the problem in
two steps: first regularize and then reconstruct. We have described two
methods based on this idea, the ASP method that is actually a particular
preconditioned iterative solver, and the ATP method that is a method that
tries to improve the approximation arising from the Tikhonov regularization.
In both cases the reconstruction is performed evaluating a matrix function
by means of the standard Arnoldi method. This idea can also be interpreted
as a modification of the iterated Tikhonov regularization (see for instance
\cite{HH} and \cite{Neu}).

Being iterative, both methods should be interpreted as methods depending on
two parameters, that is, $\lambda $ and the number of iterations $m$.
Actually our implementation of the Arnoldi method (Gram-Schmidt) is very
stable so that for a fixed $\lambda $, the undersmoothing effect,
theoretically determined by taking $m$ large, in general does not
deteriorate the approximation. Therefore the only important parameter is $%
\lambda $. Anyway, the most important property of both methods is that they
do not need an accurate estimate of this parameter to work properly (cf.
Section \ref{sec4}, Figures \ref{figure4} and \ref{figure5}, and Table 4).
Of course this property is particularly attractive for problems in which a
parameter-choice analysis is too expensive or even unfeasible as for
instance for large scale problems such as the image restoration.

As possible future developments, we observe that the ASP method could be
quite easily extended to work in connection with polynomial preconditioners
(see e.g. \cite{Bro} for a background). This can be done replacing $%
(A+\lambda I)^{-1}$ with a suitable $p_{m}(A)\approx A^{-1}$ and changing
accordingly the matrix function to evaluate. Also the symmetric version of
the ATP method (see Remark \ref{r1}) seems quite interesting and requires
further investigation.

Finally, we want to point out that the present paper was just intended to
present the basic ideas and properties of the methods; in this sense, a
reliable implementation with stopping criterium, choice of $\lambda $, etc.,
has still to be done.

\end{document}